\documentclass[10pt,reqno,a4paper]{amsart}
\usepackage{graphicx}
\usepackage[nocompress]{cite}
\usepackage{tikz-cd}

\usepackage{mathrsfs}
\usepackage{amsopn,amssymb,amsmath}
\usepackage{url}
\usepackage{subcaption}
\usepackage[hidelinks]{hyperref}

\usepackage{verbatim}
\theoremstyle{definition}
\newtheorem{theorem}{Theorem}[section]
\newtheorem{prop}[theorem]{Proposition}
\newtheorem{defn}[theorem]{Definition}
\newtheorem{cor}[theorem]{Corollary}
\newtheorem{eg}[theorem]{Example}
\newtheorem{lemma}[theorem]{Lemma}
\newtheorem{remark}[theorem]{Remark}
\numberwithin{equation}{section}

\DeclareMathOperator{\Ima}{Im}

\DeclareMathOperator{\Ext}{Ext}
\DeclareMathOperator{\LCM}{LCM}
\newcommand{\Z}{\mathbb{Z}}
\newcommand{\R}{\mathbb{R}}

\begin{document}

\title{Discrete Morse Theory for Weighted Simplicial Complexes}

\author{Chengyuan Wu\textsuperscript{*}}
\address{Department of Mathematics, National University of Singapore, Singapore 119076}
\email{wuchengyuan@u.nus.edu}
\thanks{\text{*}First authors. The project was supported in part by the Singapore Ministry of Education research grant (AcRF Tier 1 WBS No.~R-146-000-222-112). The first author was supported in part by the President's Graduate Fellowship of National University of Singapore. The second author was supported by the Postdoctoral International Exchange Program of China, 2019 project from The Office of China Postdoctoral Council, China Postdoctoral Science Foundation. The third author was supported by Natural Science Foundation of China (NSFC grant no.\ 11971144) and High-level Scientific Research Foundation of Hebei Province. The fourth author was supported by Nanyang Technological University Startup Grants M4081842, Singapore Ministry of Education Academic Research Fund Tier 1 RG31/18, Tier 2 MOE2018-T2-1-033.}

\author{Shiquan Ren\textsuperscript{*}}
\address{Yau Mathematical Sciences Center, Tsinghua University, Beijing 100084, China}
\email{srenmath@tsinghua.edu.cn}

\author{Jie Wu\textsuperscript{*}}
\address{School of Mathematical Sciences, Hebei Normal University, Hebei 050024, China}
\email{wujie@hebtu.edu.cn}

\author{Kelin Xia\textsuperscript{*}}
\address{(a) Division of Mathematical Sciences, School of Physical and Mathematical Sciences, Nanyang Technological
University, Singapore 637371 \\
(b) School of Biological Sciences, Nanyang Technological University, Singapore 637371}
\email{xiakelin@ntu.edu.sg}


%    General info
\subjclass[2010]{Primary 55N35; Secondary 55U10}

%\date{January 1, 2001 and, in revised form, June 22, 2001.}

%\dedicatory{}

\keywords{Discrete Morse Theory, Weighted Simplicial Complexes, Weighted Homology, Algebraic topology}

\begin{abstract}
In this paper, we study Forman's discrete Morse theory in the context of weighted homology. We develop weighted versions of classical theorems in discrete Morse theory. A key difference in the weighted case is that simplicial collapses do not necessarily preserve weighted homology. We work out some sufficient conditions for collapses to preserve weighted homology, as well as study the effect of elementary removals on weighted homology. An application to sequence analysis is included, where we study the weighted ordered complexes of sequences.
\end{abstract}

\maketitle
%%%%%Main Text Start
\section{Introduction}
In 1995, Robin Forman introduced discrete Morse theory in the seminal paper \cite{forman1998morse}. Subsequently, a expository user's guide to discrete Morse theory was written by Forman in \cite{forman2002user}. Since then, there have been numerous applications of discrete Morse theory in a wide range of subjects \cite{lewiner2004applications,farley2005discrete,robins2011theory,gunther2012efficient}. A main theorem of discrete Morse theory \cite[p.~10]{forman2002user} allows us to reduce the number of cells in a CW complex, while preserving its homotopy type (and hence its homology).

The weighted homology of simplicial complexes was first studied by Robert J.\ Dawson in \cite{Dawson1990}, and subsequently generalized by S.\ Ren, C.\ Wu and J.\ Wu in \cite{ren2018weighted,ren2017further}. Weighted homology can be incorporated into persistent homology to analyze weighted data \cite{ren2018weighted}. In \cite{wu2018weighted}, weighted (co)homology and weighted Laplacian was studied by C.\ Wu, S.\ Ren, J.\ Wu and K.\ Xia, with applications to biomolecules and network motifs.

With the help of suitable weights, weighted homology is able to distinguish between simplicial complexes that are homotopy equivalent. We illustrate this in Example \ref{eg:nopreserve}. Classically, constructions in algebraic topology such as the homology or homotopy functors are designed to be homotopy invariants, meaning that they do not distinguish between spaces that are homotopy equivalent. In applications, simplicial complexes which are homotopy equivalent may have different meanings. For instance, in the context of collaboration networks \cite{Carstens2013,newman2001structure}, a 2-simplex may represent 3 scientists A, B, C where each pair of scientists have a joint 2-author paper, and furthermore all three scientists have a joint 3-author paper. On the other hand, a 0-simplex may represent a single scientist with a 1-author paper. Hence, despite being homotopy equivalent, the 2-simplex and the 0-simplex have quite different meanings in this case. Weighted homology can supplement traditional topological data analysis methods \cite{bubenik2015statistical,zomorodian2012topological} by giving an option to distinguish between homotopy equivalent simplicial complexes when necessary.

In this paper, we combine the concepts of discrete Morse theory with the theory of weighted homology. The goal is to develop weighted versions of classical theorems in discrete Morse theory.

In Section \ref{sec:weightedhom}, we give a brief summary of weighted homology. In Section \ref{sec:collapse}, we study collapses of weighted simplicial complexes and their effect on weighted homology. In Theorem \ref{thm:sameweight}, we give a sufficient condition for collapses to preserve weighted homology, which will help in our subsequent study of weighted discrete Morse theory in Section \ref{sec:wdmorse}. As an application, we apply weighted discrete Morse theory to the study of sequences (including DNA/RNA sequences) via weighted ordered complexes in Section \ref{sec:sequence}.

\subsection{Related Work}

In \cite{jollenbeck2009minimal}, M.~J\"{o}llenbeck and V.~Welker study Forman's discrete Morse theory from an algebraic viewpoint. An analogous theory was
developed independently by E.~Sk\"{o}ldberg \cite{skoldberg2006morse}. The authors in \cite{jollenbeck2009minimal} consider chain complexes $C_\bullet=(C_i,\partial_i)_{i\geq 0}$ of free modules $C_i$ over a ring $R$. Subsequently, the complex $C_\bullet$ is viewed as a directed weighted graph, where the vertex set is given by the chosen basis of $C_\bullet$. The weight of the edge $e: c\to c'$ is then given by the nonzero coefficient $[c:c']\in R$ in the differential of $c$. This has some similarities to our definition of weighted boundary map in Definition \ref{def:wboundary}. Overall, the content and focus of \cite{jollenbeck2009minimal,skoldberg2006morse} is significantly different from our paper.

In \cite{salvetti2013combinatorial}, M.~Salvetti and A.~Villa study the twisted cohomology of Artin groups using discrete Morse theory. The theory is further developed in \cite{paolini2017weighted}. They define a weighted sheaf $(K,R,w)$ where $R$ is a ring and $w$ is a map between the face poset of $K$ and the ring $R$ satisfying a divisibility relation in $R$: $\sigma\subset\tau\implies w(\sigma)\mid w(\tau)$. This is the same as our definition of a weighted simplicial complex (Definition \ref{wscdef}). In their paper, the main object of study is the homology $H_*(L_*)$ of an algebraic complex $L_*(K)$ \cite[p.~1160]{salvetti2013combinatorial}. Their usage of the weights $w(\sigma)$ is in the definition of $L_*(K)$. In our paper, the main object of study is the weighted homology of the simplicial complex $H_*(K,w)$, while our usage of weights $w(\sigma)$ is in the definition of the weighted boundary map (Definition \ref{def:wboundary}). Hence, our paper is significantly different from \cite{salvetti2013combinatorial,paolini2017weighted} with respect to the main object of study, the usage of weights, and the general context of the paper. It is interesting to remark that the condition for weighted matching in \cite[p.~1158]{salvetti2013combinatorial} that two matched elements must have the same weight resembles our condition for elementary collapses to preserve weighted homology (Theorem \ref{thm:sameweight}).

Other papers involving usage of weights and discrete Morse theory include \cite{du2018rgb}, where weights are applied to different colors in the Red-Green-Blue (RGB) encoding. Discrete Morse theory is then used in combination with persistent homology for data analysis. In \cite{reininghaus2010tadd}, discrete Morse theory is used to extract the extremal structure of scalar and vector fields on 2D manifolds embedded in $\R^3$. Weights $\omega: E\to\R$ are assigned to the edges of the cell graph, followed by computing the sequence of maximum weight matchings. An algorithmic pipeline computes a hierarchy of extremal structures, where the hierarchy is defined by an
importance measure and enables the user to select an appropriate level of detail.

\section{Weighted Homology}
\label{sec:weightedhom}
In this section, we outline the main definitions and results in weighted homology. Weighted homology of simplicial complexes, together with their categorical properties, was first studied by Robert J.~Dawson \cite{Dawson1990}, where weights take integer values. In \cite{ren2018weighted}, the authors generalize the weights to take values in an integral domain $R$. It should be remarked that weighted homology is a generalization of the usual simplicial homology. When all weights are equal and nonzero, the weighted homology is isomorphic to the usual simplicial homology (see Proposition \ref{prop:constantweighting}). 

Following the context in \cite[p.~2672]{ren2018weighted}, we require $R$ to be an integral domain (with unity) when discussing weighted homology.

\begin{defn}[cf.\ {\cite[p.~229]{Dawson1990},\cite[p.~2666]{ren2018weighted}}] 
\label{wscdef}
Let $R$ be an integral domain. A \emph{weighted simplicial complex} (or \emph{WSC} for short) is  a pair $(K,w)$ consisting of a simplicial complex $K$ and a  function $w: K\to R$, such that for any $\sigma_1, \sigma_2\in K$ with $\sigma_1\subseteq \sigma_2$, we have $w(\sigma_1)\mid w(\sigma_2)$.
\end{defn}

\begin{defn}[cf.\ {\cite[p.~2673]{ren2018weighted}}]
\label{chaingroup}
Let $R$ be an integral domain. Let $C_n(K,w)$ be the free $R$-module with basis the $n$-simplices of $K$ with nonzero weight. Elements of $C_n(K,w)$, called $n$\emph{-chains}, are finite formal sums $\sum_\alpha n_\alpha\sigma_\alpha$ with coefficients $n_\alpha\in R$ and $\sigma_\alpha\in K$.
\end{defn}

\begin{defn}[cf.\ {\cite[p.~234]{Dawson1990},\cite[p.~2674]{ren2018weighted}}]
\label{def:wboundary}
The \emph{weighted boundary map} $\partial_n^w: C_n(K,w)\to C_{n-1}(K,w)$ is the map (extended $R$-linearly): \[\partial_n^w(\sigma)=\sum_{i=0}^n\frac{w(\sigma)}{w(d_i(\sigma))}(-1)^id_i(\sigma)\] where the \emph{face maps} $d_i$ are defined as: \[d_i(\sigma)=[v_0,\dots,\widehat{v_i},\dots,v_n]\qquad\text{(deleting the vertex $v_i$)}\]  for any $n$-simplex $\sigma=[v_0,\dots,v_n]$. 
\end{defn}

\begin{defn}[cf.\ {\cite[p.~2677]{ren2018weighted}}]
We define the $n$th weighted homology group with coefficients in $R$ by
\begin{equation}
H_n(K,w;R):=\ker(\partial_n^w)/\Ima(\partial_{n+1}^w),
\end{equation}
where $\partial_n^w$ is the weighted boundary map defined in Definition \ref{def:wboundary}. For convenience, if there is no danger of confusion, we may simply write $H_n(K,w)$ to denote $H_n(K,w;R)$.
\end{defn}

\begin{prop}[{cf.\ \cite[p.~239]{Dawson1990},\cite[p.~2679]{ren2018weighted}}]
\label{prop:constantweighting}
For the constant weighting $(K,w)$, $w(\sigma)\equiv a\in R\setminus\{0\}$ for all $\sigma\in K$, the weighted homology functor is the same as the standard simplicial homology functor.
\qed
\end{prop}

A sample calculation of weighted homology can be found in \cite[p.~2678]{ren2018weighted}. Weighted homology can be effectively computed by the Smith Normal Form algorithm (cf.\ \cite{dumas2003computing}) for the weighted boundary matrices.

\section{Collapses of Weighted Simplicial Complexes}
\label{sec:collapse}
The notion of simplicial collapse was first introduced by J.H.C.\ Whitehead in \cite{whitehead1939simplicial}. Subsequently, simplicial collapse is noted to play a fundamental role in discrete Morse theory \cite[p.~12]{forman2002user}. In this section, we study collapses of weighted simplicial complexes and their effects on weighted homology. A key difference in the weighted case is that collapses do not necessarily preserve weighted homology.

Let $K$ be a finite (abstract) simplicial complex. We denote a $n$-dimensional simplex $\sigma\in K$ by $\sigma^{(n)}$.

\begin{defn}[cf.\ {\cite{forman2002user,forman1998morse}}]
For simplices $\sigma,\tau\in K$, we write $\sigma<\tau$ to indicate that $\sigma$ is a \emph{proper face} of $\tau$, that is, $\sigma\subsetneq \tau$. 

If $\sigma<\tau$ and $\sigma$ is not a proper face of any other simplex of $K$ (other than $\tau$), we say that $\sigma$ is a \emph{free face} of $\tau$ (in $K$).
\end{defn}

\begin{remark}
We observe that if $\sigma$ is a free face of $\tau$, then necessarily $\dim\tau=\dim\sigma+1$ and $\tau$ is a maximal face of $K$.
\end{remark}

\begin{defn}[{cf.\ \cite[p.~99]{forman1998morse}}]
Let $K$ be a simplicial complex and $\sigma^{(n-1)}<\tau^{(n)}$ be two simplices of $K$ such that $\sigma$ is a free face of $\tau$.

Let $L=K\setminus \{\sigma, \tau\}$ be the simplicial complex resulting from deleting $\sigma$ and $\tau$ from $K$. We say that $K$ collapses onto $L$ by an \emph{elementary collapse} of dimension $n$. More generally, we say $K$ collapses onto $L$, denoted by $K\searrow L$, if $K$ can be transformed into $L$ by a finite sequence of elementary collapses.
\end{defn}

\begin{remark}
\label{remark:conven}
For convenience, we may write $(K,w)\searrow (L,w)$ to denote a collapse of the WSC $(K,w)$ onto the WSC $(L,w|_L)$. Similarly, if $K\searrow L$ is a collapse, we may write $(K,w)$ and $(L,w)$ to mean that the weight function on $L$ is the restriction of the weight function on $K$.
\end{remark}

\begin{defn}
Let $(K,w)\searrow (L,w)$ be a collapse of WSCs. We say that $K\searrow L$ \emph{preserves weighted homology} if
\begin{equation*}
H_i(K,w)\cong H_i(L,w)
\end{equation*}
for all $i\geq 0$.
\end{defn}

We show some examples of collapses and their effects on weighted homology.

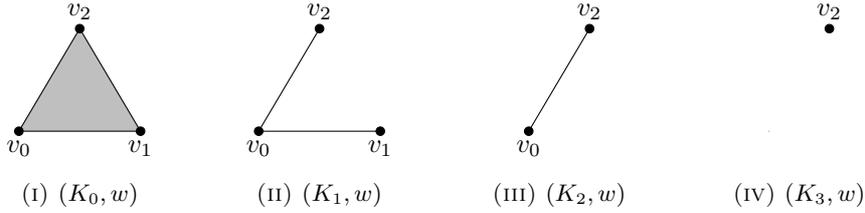
\begin{figure}[htbp]
\begin{subfigure}{0.24\textwidth}
\centering
\begin{tikzpicture}[scale=0.8]
\fill[fill=lightgray]
(0,0)  
-- (2,0)
-- (1,1.7);
\filldraw 
(0,0) circle (2pt) node[align=left,below] {$v_0$}
(2,0) circle (2pt) node[align=left,below] {$v_1$}  
(1,1.7) circle (2pt) node[align=left,above] {$v_2$};
\draw (0,0)--(2,0);
\draw (2,0)--(1,1.7);
\draw (1,1.7)--(0,0);
\end{tikzpicture}
\caption{$(K_0,w)$}
\label{subfig:2simplex}
\end{subfigure}
\begin{subfigure}{0.24\textwidth}
\centering
\begin{tikzpicture}[scale=0.8]
\filldraw 
(0,0) circle (2pt) node[align=left,below] {$v_0$}
(2,0) circle (2pt) node[align=left,below] {$v_1$}  
(1,1.7) circle (2pt) node[align=left,above] {$v_2$};
\draw (0,0)--(2,0);
%\draw (2,0)--(1,1.7);
\draw (1,1.7)--(0,0);
\end{tikzpicture}
\caption{$(K_1,w)$}
\end{subfigure}
\begin{subfigure}{0.24\textwidth}
\centering
\begin{tikzpicture}[scale=0.8]
\filldraw 
(0,0) circle (2pt) node[align=left,below] {$v_0$}
%(2,0) circle (2pt) node[align=left,below] {$v_1$}  
(1,1.7) circle (2pt) node[align=left,above] {$v_2$};
%\draw (0,0)--(2,0);
%\draw (2,0)--(1,1.7);
\draw (1,1.7)--(0,0);
\end{tikzpicture}
\caption{$(K_2,w)$}
\end{subfigure}
\begin{subfigure}{0.24\textwidth}
\centering
\begin{tikzpicture}[scale=0.8]
\filldraw 
(0,0) circle (0pt) node[align=left,below] {\phantom{$v_0$}}
%(2,0) circle (2pt) node[align=left,below] {$v_1$}  
(1,1.7) circle (2pt) node[align=left,above] {$v_2$};
%\draw (0,0)--(2,0);
%\draw (2,0)--(1,1.7);
%\draw (1,1.7)--(0,0);
\end{tikzpicture}
\caption{$(K_3,w)$}
\label{subfig:0simplex}
\end{subfigure}
\caption{The four WSCs $(K_i,w)$ for $0\leq i\leq 3$, with weights as described in Example \ref{eg:nopreserve}. The WSCs are chosen such that ${K_i\searrow K_{i+1}}$ is an elementary collapse for $0\leq i\leq 2$.}
\label{fig:collapses}
\end{figure}

\begin{eg}
\label{eg:nopreserve}
Consider the WSCs in Figure \ref{fig:collapses}. Let $w:K_0\to\Z$ be defined by
\begin{align*}
w([v_0])=w([v_1])=1,\quad w([v_2])&=2,\\
w([v_0,v_1])=w([v_0,v_2])=2,\quad w([v_1,v_2])&=4,\\
w([v_0,v_1,v_2])&=4.
\end{align*}

For convenience, we also use $w$ to denote the weight functions $w|_{K_i}:K_i\to\Z$ for $1\leq i\leq 3$ (see Remark \ref{remark:conven}). We calculate the integral weighted homology of the WSCs $(K_i,w)$ and summarize them in Table \ref{table:weightedhom}.

\begin{table}[htbp]
\centering
\caption{The table lists the integral weighted homology of the WSCs $(K_i,w)$. For instance, $H_0(K_0,w)=\Z\oplus\Z/2$, $H_1(K_0,w)=H_2(K_0,w)=0$.}
  \begin{tabular}{| c | c | c | c | c | }
    \hline
     & $(K_0,w)$ & $(K_1,w)$ & $(K_2,w)$ & $(K_3,w)$ \\ \hline
     $H_0$ & $\Z\oplus\Z/2$ & $\Z\oplus\Z/2$ & $\Z$ & $\Z$ \\ \hline
     $H_1$ & $0$ & $0$ & $0$ & $0$ \\ \hline
     $H_2$ & $0$ & $0$ & $0$ & $0$ \\ \hline
  \end{tabular}
\label{table:weightedhom}
\end{table}

We see that certain elementary collapses, such as $K_0\searrow K_1$ and $K_2\searrow K_3$, preserve the weighted homology while $K_1\searrow K_2$ does not.
\end{eg}

A natural question would be to ask for which cases do elementary collapses preserve weighted homology. We work out some sufficient conditions for elementary collapses to preserve weighted homology.

\begin{theorem}
\label{thm:sameweight}
Let $(K,w)\searrow (L,w)$ be an elementary collapse of dimension $n\geq 1$, where $L=K\setminus\{\sigma^{(n-1)},\tau^{(n)}\}$.

Suppose that $w(\sigma)=w(\tau)=a\in R\setminus\{0\}$. That is, $\sigma$ and $\tau$ have the same nonzero weight.

Then, $K\searrow L$ preserves weighted homology.
\end{theorem}
\begin{proof}
By Proposition 4.7 in \cite[p.~2676]{ren2018weighted}, the inclusion map $i: L\to K$ induces a chain map $i_\sharp: C_k(L,w)\to C_k(K,w)$. We define the quotient module
\begin{equation*}
D_k:=C_k(K,w)/C_k(L,w).
\end{equation*}

We have the following short exact sequence of chain complexes:
\begin{equation*}
0\to C_*(L,w)\to C_*(K,w)\to D_*\to 0.
\end{equation*}
The boundary operator of $C_*(K,w)$ is the weighted boundary map $\partial^w$. The boundary operators of $C_*(L,w)$ and $D_*$, denoted by $\partial^L$ and $\partial^D$, are canonically induced from $\partial^w$ by the restriction and quotient respectively. To be precise, $\partial^L(\sigma)=\partial^w(i(\sigma))$ for $\sigma\in C_*(L,w)$ and $\partial^D([\sigma])=[\partial^w(\sigma)]$ for the equivalence class $[\sigma]\in D_*$.

By the Zig-zag lemma \cite[p.~136]{Munkres1984}, there is a long exact sequence
\begin{equation}
\label{eq:longexact}
\dots\to H_{k+1}(D_*)\to H_k(L,w)\to H_k(K,w)\to H_k(D_*)\to H_{k-1}(L,w)\to \dots
\end{equation}

We observe that 
\begin{equation*}
D_k\cong\begin{cases}
R &\text{if $k=n-1$ or $k=n$,}\\
0 &\text{otherwise.}
\end{cases}
\end{equation*}

Hence, it is clear that $\ker\left(\partial_k^D: D_k\to D_{k-1}\right)=0$ for $k\notin\{n-1,n\}$. Consequently, we have $H_k(D_*)=0$ for $k\notin\{n-1,n\}$. 

We note that $\ker\left(\partial_{n-1}^D: D_{n-1}\to D_{n-2}\right)\cong\langle [\sigma]\rangle$. Since
\begin{equation}
\label{eq:boundarytau}
\begin{split}
\partial_n^D([\tau])&=[\partial_n^w(\tau)]\\
&=[\pm\frac{w(\tau)}{w(\sigma)}\sigma]\\
&=\pm[\sigma],
\end{split}
\end{equation}
we also have $\Ima(\partial_n^D)\cong\langle [\sigma]\rangle$. Thus, $H_{n-1}(D_*)\cong 0$.

The calculations in \eqref{eq:boundarytau} also implies $\ker(\partial_n^D)=0$, hence $H_n(D_*)=0$. Essentially, we have shown that $H_k(D_*)\cong 0$ for all $k$.

Hence, the long exact sequence in \eqref{eq:longexact} is of the form
\begin{equation*}
\dots\to 0\to H_k(L,w)\to H_k(K,w)\to 0\to\dots
\end{equation*}
which implies that $H_k(L,w)\cong H_k(K,w)$ for all $k$.
\end{proof}

\begin{remark}
Theorem \ref{thm:sameweight} explains why $K_0\searrow K_1$ preserves the weighted homology in Example \ref{eg:nopreserve}.

The converse of Theorem \ref{thm:sameweight} is not true: If an elementary collapse $K\searrow L$ preserves weighted homology, it does not imply that the two removed simplices $\sigma, \tau$ have the same nonzero weight. A counterexample is $K_2\searrow K_3$ in Example \ref{eg:nopreserve}.
\end{remark}

A corollary to Theorem \ref{thm:sameweight} is an alternative proof that a collapse $K\searrow L$ always preserves the usual simplicial homology, without using the fact that $K$ and $L$ are homotopy equivalent.

\begin{cor}
Let $K\searrow L$ be a collapse (not necessarily elementary). Then, $H_*(K)\cong H_*(L)$.
\end{cor}
\begin{proof}
It suffices to prove the statement for an elementary collapse $K\searrow L$, and extend to general collapses by induction.

By Proposition \ref{prop:constantweighting}, the usual homology $H_*(K)$ is isomorphic to the weighted homology $H_*(K,w)$ when $w$ is the constant weighting $w(\sigma)\equiv a\in R\setminus\{0\}$. Similarly, $H_*(L)\cong H_*(L,w)$.

By Theorem \ref{thm:sameweight}, $H_*(K,w)\cong H_*(L,w)$ since all weights of simplices are the same nonzero element $a$. Therefore,
\begin{equation*}
H_*(K)\cong H_*(K,w)\cong H_*(L,w)\cong H_*(L).
\end{equation*}
\end{proof}

We also note that Theorem \ref{thm:sameweight} can be slightly strengthened: the two removed simplices $\sigma^{(n-1)}, \tau^{(n)}$ need only to have weights that are associates in $R$, not necessarily equal. We state this more precisely in the theorem below.

\begin{theorem}
\label{thm:strongerassoc}
Let $(K,w)\searrow (L,w)$ be an elementary collapse of dimension $n\geq 1$, where $L=K\setminus\{\sigma^{(n-1)},\tau^{(n)}\}$.

Suppose that $w(\sigma)\in R\setminus\{0\}$ and $w(\tau)=uw(\sigma)$ for some unit $u\in R$. That is, $w(\sigma)$ and $w(\tau)$ are nonzero associates.

Then, $K\searrow L$ preserves weighted homology.
\end{theorem}
\begin{proof}
The proof is essentially the same as that of Theorem \ref{thm:sameweight}. There is a minor difference in that Equation \eqref{eq:boundarytau} is replaced by $\partial_n^D([\tau])=\pm u[\sigma]$. This does not affect the subsequent result that $\Ima(\partial_n^D)\cong\langle [\sigma]\rangle$.
\end{proof}

\begin{remark}
Although Theorem \ref{thm:strongerassoc} is stronger than Theorem \ref{thm:sameweight}, it is considerably easier to check the condition $w(\sigma)=w(\tau)$ in Theorem \ref{thm:sameweight}. For instance, in the ring $R=\Z[\sqrt{2}]$, it may not be immediately obvious that $2+\sqrt{2}$ and $4+3\sqrt{2}$ are associates. Hence, our subsequent theory of weighted discrete Morse theory will be based on Theorem \ref{thm:sameweight} instead of Theorem \ref{thm:strongerassoc}.
\end{remark}

\subsection{Elementary Removals and Weighted Homology}
We now study an operation similar to elementary collapses. Let $K$ be a simplicial complex. If we remove a maximal face $\sigma$ from $K$, we note that the result $L=K\setminus\{\sigma\}$ is still a simplicial complex. We will call this operation of removing a maximal simplex an \emph{elementary removal}, similar to the notation in \cite[p.~2]{zorn2009discrete}. Unlike elementary collapses, it is not guaranteed that $K$ and $L=K\setminus\{\sigma\}$ are homotopy equivalent. Hence, elementary removals may not preserve usual homology, let alone weighted homology. In this subsection, we study the effects of elementary removals on weighted homology. Subsequently, the results will be useful when we study weighted discrete Morse theory in Section \ref{sec:wdmorse}.

\begin{theorem}
\label{thm:removalhom}
Let $(K,w)$ be a WSC. Let $\sigma^{(n)}$ be a maximal face of $K$ with $w(\sigma)\neq 0$, and let $L=K\setminus\{\sigma\}$.

Then, we have the following results:
\begin{enumerate}
\item \begin{equation*}
H_k(L,w)\cong H_k(K,w)
\end{equation*}
for $k\notin\{n-1,n\}$.

\item We also have
\begin{equation*}
H_{n-1}(K,w)\cong H_{n-1}(L,w)/\{r[\partial^w\sigma]\mid r\in R\},
\end{equation*}
where $\partial^w$ denotes the weighted boundary operator of $C_*(K,w)$, and $[\partial^w\sigma]$ denotes the homology class in $H_{n-1}(L,w)$.

\item Let $i_*:H_n(L,w)\to H_n(K,w)$ be the map canonically induced by the inclusion $i: C_*(L,w)\to C_*(K,w)$. Then, we have 
\begin{equation*}
H_n(K,w)/\Ima(i_*:H_n(L,w)\to H_n(K,w))\cong\{r\in R\mid r[\partial^w\sigma]=0\}.
\end{equation*}
\end{enumerate}
\end{theorem}
\begin{proof}
Similar to the proof in Theorem \ref{thm:sameweight}, we define $D_k:=C_k(K,w)/C_k(L,w)$. Consider the short exact sequence of chain complexes:
\begin{equation*}
0\to C_*(L,w)\xrightarrow{i}C_*(K,w)\xrightarrow{\pi}D_*\to 0,
\end{equation*}
where the chain maps $i$, $\pi$ are the canonical inclusion and projection respectively.

We note that
\begin{equation*}
D_k\cong\begin{cases}
R &\text{if $k=n$,}\\
0 &\text{otherwise.}
\end{cases}
\end{equation*}

By the Zig-zag Lemma, there is a long exact sequence
\begin{equation}
\label{eq:zigzag2}
\dots\to H_{k+1}(D_*)\xrightarrow{\partial^w_*} H_k(L,w)\xrightarrow{i_*}H_k(K,w)\xrightarrow{\pi_*}H_k(D_*)\xrightarrow{\partial^w_*}H_{k-1}(L,w)\to \dots
\end{equation}
where $\partial^w_*$ is induced by the weighted boundary operator in $C_*(K,w)$.

For $k\neq n$, it is clear that $H_k(D_*)=0$. If $k\notin\{n-1,n\}$, then $H_{k+1}(D_*)=0$ and $H_k(D_*)=0$. Hence, we can conclude from the long exact sequence \eqref{eq:zigzag2} that $H_k(L,w)\cong H_k(K,w)$ for $k\notin\{n-1,n\}$.

When $k=n$, we have $H_n(D_*)\cong R$ as $R$-modules. Now, consider the exact sequence
\begin{equation*}
0\to H_n(L,w)\xrightarrow{i_*}H_n(K,w)\xrightarrow{\pi_*}R\xrightarrow{\partial^w_*}H_{n-1}(L,w)\xrightarrow{i_*}H_{n-1}(K,w)\to 0.
\end{equation*}

We note that the map $R\xrightarrow{\partial^w_*}H_{n-1}(L,w)$ is defined by $\partial^w_*([\sigma])=[\partial^w\sigma]$ and extended $R$-linearly. By the first isomorphism theorem and exactness, we have
\begin{equation*}
\begin{split}
H_{n-1}(K,w)&\cong H_{n-1}(L,w)/\ker(i_*:H_{n-1}(L,w)\to H_{n-1}(K,w))\\
&=H_{n-1}(L,w)/\{r[\partial^w\sigma]\mid r\in R\}.
\end{split}
\end{equation*}

Similarly, we have
\begin{equation*}
\begin{split}
&H_n(K,w)/\Ima(i_*:H_n(L,w)\to H_n(K,w))\\
&=H_n(K,w)/\ker(\pi_*:H_n(K,w)\to R)\\
&\cong\Ima(\pi_*:H_n(K,w)\to R)\\
&=\ker(\partial^w_*:R\to H_{n-1}(L,w))\\
&\cong\{r\in R\mid r[\partial^w\sigma]=0\}.
\end{split}
\end{equation*}
\end{proof}

In the case where $R$ is a PID, for instance $R=\mathbb{Z}$, the result (3) in Theorem \ref{thm:removalhom} takes on a neater form (3').

\begin{cor}
\label{cor:removalhom}
Let $R$ be a PID. Let $(K,w)$ be a WSC, where $w:K\to R$ is the weight function. Let $\sigma^{(n)}$ be a maximal face of $K$ with $w(\sigma)\neq 0$, and let $L=K\setminus\{\sigma\}$.

We have:
\begin{itemize}
\item [(3')] If $[\partial^w\sigma]\in H_{n-1}(L,w)$ is a torsion element, i.e.\ there exists $r\in R\setminus\{0\}$ such that $r[\partial^w\sigma]=0$, then
\begin{equation*}
H_n(K,w)\cong H_n(L,w)\oplus R.
\end{equation*}

If $[\partial^w\sigma]$ is not a torsion element, then $H_n(K,w)\cong H_n(L,w)$.
\end{itemize}
\end{cor}
\begin{proof}
We note that $I:=\{r\in R\mid r[\partial^w\sigma]=0\}$ is an ideal of $R$. Since $R$ is a PID, hence $I=(a)$ for some $a\in R$. By the first isomorphism theorem and exactness, we have 
\begin{equation*}
\begin{split}
J&:=\Ima(i_*:H_n(L,w)\to H_n(K,w))\\
&\cong H_n(L,w)/\ker(i_*:H_n(L,w)\to H_n(K,w))\\
&\cong H_n(L,w).
\end{split}
\end{equation*}

We note that $I$ is a free $R$-module, with basis $\{a\}$ if $a\neq 0$, and with basis $\emptyset$ if $a=0$. In particular, $I$ is a projective $R$-module and hence $\Ext_R^1(I,J)=0$. By Theorem \ref{thm:removalhom} (3), we have the short exact sequence
\begin{equation*}
0\to J\to H_n(K,w)\to I\to 0.
\end{equation*}

Since $\Ext_R^1(I,J)=0$, we conclude that
\begin{equation*}
\begin{split}
H_n(K,w)&\cong J\oplus I\\
&\cong H_n(L,w)\oplus I.
\end{split}
\end{equation*}

If $[\partial^w\sigma]$ is a torsion element, then $a\neq 0$. We have $I\cong R$ as $R$-modules, where the isomorphism is given by $a\mapsto 1$. Hence, $H_n(K,w)\cong H_n(L,w)\oplus R$.

If $[\partial^w\sigma]$ is not a torsion element, then $I=0$ and therefore $H_n(K,w)\cong H_n(L,w)$.
\end{proof}

We show an example of the effect of an elementary removal on the weighted homology.

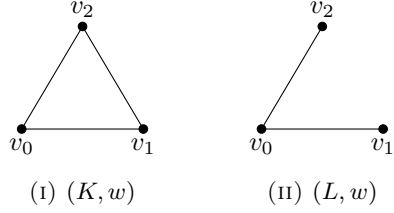
\begin{figure}[htbp]
\begin{subfigure}{0.24\textwidth}
\centering
\begin{tikzpicture}[scale=0.8]
\filldraw 
(0,0) circle (2pt) node[align=left,below] {$v_0$}
(2,0) circle (2pt) node[align=left,below] {$v_1$}  
(1,1.7) circle (2pt) node[align=left,above] {$v_2$};
\draw (0,0)--(2,0);
\draw (2,0)--(1,1.7);
\draw (1,1.7)--(0,0);
\end{tikzpicture}
\caption{$(K,w)$}
\label{subfig:2simplex}
\end{subfigure}
\begin{subfigure}{0.24\textwidth}
\centering
\begin{tikzpicture}[scale=0.8]
\filldraw 
(0,0) circle (2pt) node[align=left,below] {$v_0$}
(2,0) circle (2pt) node[align=left,below] {$v_1$}  
(1,1.7) circle (2pt) node[align=left,above] {$v_2$};
\draw (0,0)--(2,0);
%\draw (2,0)--(1,1.7);
\draw (1,1.7)--(0,0);
\end{tikzpicture}
\caption{$(L,w)$}
\end{subfigure}
\caption{The first simplicial complex $K$ consists of 3 edges and 3 vertices. The second simplicial complex $L$ is obtained by $L:=K\setminus\{\sigma\}$, where $\sigma^{(1)}=[v_1,v_2]$ is a maximal face of $K$.}
\label{fig:eleremoval}
\end{figure}

\begin{eg}
Consider the WSCs in Figure \ref{fig:eleremoval} with the following weight function $w:K\to\Z$:
\begin{align*}
w([v_0])=w([v_1])=w([v_2])=1,\\
w([v_0,v_1])=w([v_0,v_2])=w([v_1,v_2])=2.
\end{align*}

The weight function for $L$ is the restriction of $w$ to $L$, also denoted $w$ for convenience.

We first calculate the weighted homology groups for $(K,w)$. We have
\begin{align*}
\ker\partial_0^w&=\langle v_0,v_1,v_2\rangle\\
\Ima\partial_1^w&=\langle 2v_1-2v_0,2v_2-2v_0,2v_2-2v_1\rangle\\
\ker\partial_1^w&=\langle [v_0,v_1]-[v_0,v_2]+[v_1,v_2]\rangle,
\end{align*}

which implies that
\begin{align*}
H_0(K,w)&\cong\Z\oplus\Z/2\oplus\Z/2\\
H_1(K,w)&\cong\Z.
\end{align*}

Similarly, we can calculate that
\begin{align*}
H_0(L,w)&\cong\Z\oplus\Z/2\oplus\Z/2\\
H_1(L,w)&\cong 0.
\end{align*}

Now, we show that Theorem \ref{thm:removalhom} and Corollary \ref{cor:removalhom} agrees with our calculations. For the maximal face $\sigma^{(1)}=[v_1,v_2]$, we calculate that
\begin{equation*}
\begin{split}
\partial^w\sigma&=2v_2-2v_1\\
&=(2v_2-2v_0)-(2v_1-2v_0)\\
&\in\Ima(\partial_1^w: C_1(L,w)\to C_0(L,w)).
\end{split}
\end{equation*}

Hence, $[\partial^w\sigma]=0$ as a homology class in $H_0(L,w)$. Therefore, by Theorem \ref{thm:removalhom}~(2), we have
\begin{equation*}
H_0(K,w)\cong H_0(L,w)/0\cong H_0(L,w)
\end{equation*}
which agrees with our calculations.

By Corollary \ref{cor:removalhom}, since $[\partial^w\sigma]$ is a torsion element (it is annihilated by 1), hence
\begin{equation*}
H_1(K,w)\cong H_1(L,w)\oplus\Z.
\end{equation*}
Again, this agrees with our earlier calculations.
\end{eg}

\section{Weighted Discrete Morse Theory}
\label{sec:wdmorse}
In this section, we study and develop a weighted version of discrete Morse theory. In view of Proposition \ref{prop:constantweighting}, we may identify an unweighted simplicial complex $K$ with a WSC $(K,w)$ with \emph{nonzero constant weighting} $w(\sigma)\equiv a\in R\setminus\{0\}$ for all $\sigma\in K$. All theorems and definitions in this section reduce to the classical cases when the weight function $w: K\to R$ is the nonzero constant weighting.

\begin{defn}[cf.\ {\cite[p.~9]{forman2002user}}]
\label{defn:weightedmfunction}
Let $K$ be a simplicial complex. A function $f:K\to\R$ is a \emph{discrete Morse function} if for every $\alpha^{(n)}\in K$, the following two conditions are both satisfied:

\begin{enumerate}
\item $\left| \{\beta^{(n+1)}>\alpha\mid f(\beta)\leq f(\alpha)\} \right|\leq1$, and
\item $\left| \{\gamma^{(n-1)}<\alpha\mid f(\gamma)\geq f(\alpha)\} \right|\leq 1$.
\end{enumerate}
\end{defn}

Let $(K,w)$ be a WSC with a discrete Morse function $f$. We define critical and $w$-simple simplices.

\begin{defn}[cf.\ {\cite[p.~10]{forman2002user}}]
A simplex $\alpha^{(n)}$ is \emph{critical} if

\begin{enumerate}
\item $\left| \{\beta^{(n+1)}>\alpha\mid f(\beta)\leq f(\alpha)\} \right|=0$, and
\item $\left| \{\gamma^{(n-1)}<\alpha\mid f(\gamma)\geq f(\alpha)\} \right|=0$.
\end{enumerate}
\end{defn}

\begin{defn}[$w$-simple simplex]
A simplex $\alpha^{(n)}$ is said to be \emph{$w$-simple} if $w(\alpha)\neq 0$ and for all $\gamma^{(n-1)}<\alpha$, we have
\begin{equation*}
f(\gamma)\geq f(\alpha)\implies w(\gamma)=w(\alpha).
\end{equation*}
\end{defn}

\begin{remark}
For the nonzero constant weighting $w(\sigma)\equiv a\in R\setminus\{0\}$, all simplices are $w$-simple.
\end{remark}

We require the following basic property of simplicial complexes.

\begin{lemma}[cf.\ {\cite[p.~98]{forman1998morse}}]
\label{lemma:middlesimplex}
Let $n\geq 1$. Suppose $\beta^{(n+1)}>\alpha^{(n)}>\gamma^{(n-1)}$, then there exists a unique $n$-simplex $\alpha'\neq\alpha$ such that $\beta>\alpha'>\gamma$.
\qed
\end{lemma}

Similar to the classical case, it follows directly from the definitions that a simplex cannot simultaneously fail both conditions in the test for criticality \cite[p.~10]{forman2002user}. 

\begin{lemma}[cf.\ {\cite[p.~10]{forman2002user}}]
\label{lemma:1critfail}
Let $K$ be a simplicial complex with a discrete Morse function $f$. Then, for any simplex $\alpha^{(n)}\in K$, either

\begin{enumerate}
\item $\left| \{\beta^{(n+1)}>\alpha\mid f(\beta)\leq f(\alpha)\} \right|=0$, or
\item $\left| \{\gamma^{(n-1)}<\alpha\mid f(\gamma)\geq f(\alpha)\} \right|=0$.
\end{enumerate}
\qed
\end{lemma}

\begin{defn}[cf.\ {\cite[p.~10]{forman2002user}}]
Let $K$ be a simplicial complex with a discrete Morse function $f$. For any $c\in\R$, we define the \emph{level subcomplex} $K(c)$ by

\begin{equation*}
K(c):=\bigcup_{\substack{\alpha\in K\\ f(\alpha)\leq c}}\Big(\bigcup_{\beta\leq\alpha}\beta\Big).
\end{equation*}

That is, $K(c)$ is the subcomplex of $K$ consisting of all simplices $\alpha\in K$ with $f(\alpha)\leq c$, together with all of their faces.
\end{defn}

We make the following basic but useful observation:
\begin{lemma}
\label{lemma:K(c)cond}
A simplex $\alpha\in K$ is in $K(c)$ if and only if $f(\alpha)\leq c$ or there exists $\beta>\alpha$ such that $f(\beta)\leq c$.
\qed
\end{lemma}

\begin{lemma}[cf.\ {\cite[p.~104]{forman1998morse}}]
\label{lemma:existsn+1}
Let $\alpha^{(n)}\in K$ and suppose $\beta>\alpha$. Then there exists a $(n+1)$-simplex $\widetilde{\beta}^{(n+1)}$ with $\alpha<\widetilde{\beta}\leq\beta$ and $f(\widetilde{\beta})\leq f(\beta)$.
\qed
\end{lemma}

The next lemma is a generalization of Lemma \ref{lemma:middlesimplex}.

\begin{lemma}
\label{lemma:middlesimpgen}
Suppose $\beta^{(n)}>\gamma^{(m)}$. For any integer $k$ such that $m\leq k< n$, there exists $\alpha^{(k)}$ such that $\beta^{(n)}>\alpha^{(k)}\geq\gamma^{(m)}$.

Furthermore, if $m<k<n$, there exists distinct $\widetilde{\alpha}^{(k)}\neq{\alpha}^{(k)}$ such that $\beta^{(n)}>\alpha^{(k)}>\gamma^{(m)}$ and $\beta^{(n)}>\widetilde{\alpha}^{(k)}>\gamma^{(m)}$.
\end{lemma}
\begin{proof}
We write $\beta=[v_0,\dots,v_n]$. We note that $\gamma$ is the simplex spanned by deleting $n-m\geq 1$ vertices $\{v_{i_1},\dots,v_{i_{n-m}}\}$ from $\{v_0,\dots,v_n\}$. We may then take $\alpha$ to be the simplex spanned by deleting $n-k\geq 1$ vertices $\{v_{j_1},\dots,v_{j_{n-k}}\}$ from $\{v_0,\dots,v_n\}$, such that $\{v_{j_1},\dots,v_{j_{n-k}}\}\subseteq\{v_{i_1},\dots,v_{i_{n-m}}\}$. Then, we have $\beta^{(n)}>\alpha^{(k)}\geq\gamma^{(m)}$.

Furthermore, if $m<k<n$, then note that $1\leq n-k<n-m$ and $n-m\geq 2$. Hence, there are ${{n-m}\choose{n-k}}\geq{2 \choose 1}=2$ ways of choosing subsets $\{v_{j_1},\dots,v_{j_{n-k}}\}\subseteq\{v_{i_1},\dots,v_{i_{n-m}}\}$, which corresponds to at least 2 distinct choices $\alpha^{(k)}$ and $\widetilde{\alpha}^{(k)}$.
\end{proof}

We now begin to generalize the main theorems of discrete Morse theory (cf.\ \cite[p.~104]{forman1998morse}) to the weighted case. The next theorem is the weighted version of Theorem 3.3 in \cite{forman1998morse}.

\begin{theorem}[cf.\ {\cite[p.~104]{forman1998morse}}]
\label{thm:maincollapse}
Let $(K,w)$ be a WSC with a discrete Morse function $f$. Let $a<b$ be real numbers. Suppose that for all simplices $\alpha\in K$, we have
\begin{equation*}
f(\alpha)\in (a,b]\implies\alpha\ \text{is non-critical and $w$-simple}.
\end{equation*}

Then, $K(b)\searrow K(a)$ is a collapse that preserves weighted homology.
\end{theorem}
\begin{proof}
If $f^{-1}((a,b])=\emptyset$, then clearly $K(a)=K(b)$ and we are done. 

We first prove the result for the case $|f^{-1}((a,b])|=1$, where there is exactly one simplex $\alpha^{(n)}$ with $f(\alpha)\in(a,b]$. By hypothesis, $\alpha$ is non-critical and $w$-simple.

By Lemma \ref{lemma:1critfail}, exactly 1 of the following is true:
\begin{enumerate}
\item There exists $\beta^{(n+1)}>\alpha$ with $f(\beta)\leq f(\alpha)$.
\item There exists $\gamma^{(n-1)}<\alpha$ with $f(\gamma)\geq f(\alpha)$.
\end{enumerate}

For case (1), we have $f(\beta)\leq a$, since $\alpha$ is the only simplex taking values in $(a,b]$ when applying the function $f$. Thus, $\beta\in K(a)$. Since $\alpha\leq\beta$, hence $\alpha$ is in $K(a)$ too. Thus, $K(a)=K(b)$ and the proof is finished.

From now on, we suppose case (2) is true. Since case (1) is not true, hence for all $\beta^{(n+1)}>\alpha$ we have $f(\beta)>f(\alpha)$, and in fact $f(\beta)>b$. Now, suppose to the contrary there exists $\sigma>\alpha$ with $f(\sigma)\leq a$. By Lemma \ref{lemma:existsn+1}, there exists $\widetilde{\beta}^{(n+1)}$ with $\alpha<\widetilde{\beta}\leq\sigma$ and $f(\widetilde{\beta})\leq f(\sigma)\leq a$. This is a contradiction, as previous arguments imply that $f(\widetilde{\beta})>b$. In short, there does not exist $\sigma>\alpha$ with $f(\sigma)\leq a$. By Lemma \ref{lemma:K(c)cond}, we can conclude that $\alpha\notin K(a)$.

Since case (2) is true, there exists $\gamma^{(n-1)}<\alpha$ with $f(\gamma)\geq f(\alpha)$, and in fact $f(\gamma)>b$. Suppose to the contrary there exists $\tau>\gamma$ with $f(\tau)\leq a$. By Lemma \ref{lemma:existsn+1}, there exists $\widetilde{\tau}^{(n)}$ with $\gamma<\widetilde{\tau}\leq\tau$ and $f(\widetilde{\tau})\leq f(\tau)\leq a$. Clearly, $\widetilde{\tau}\neq\alpha$ since $f(\alpha)\in (a,b]$. However by the definition of discrete Morse function, we have $f(\widetilde{\tau})>f(\gamma)>b$. This is a contradiction, and by Lemma \ref{lemma:K(c)cond}, we can conclude that $\gamma\notin K(a)$.

Next, we show that $\gamma^{(n-1)}$ is a free face of $\alpha^{(n)}$ (in $K(b)$). Suppose to the contrary there exists a simplex $S\in K(b)$ with $S>\gamma$ and $S\neq \alpha$. Note that it is not possible that $f(S)\leq b$, since $f(S)\in (a,b]$ contradicts $S\neq\alpha$ and $f(S)\leq a$ contradicts $\gamma\notin K(a)$. Thus, by Lemma \ref{lemma:K(c)cond}, there exists a simplex $T>S$ with $f(T)\leq b$. Clearly, $T\neq\alpha$ since $\dim T>\dim S\geq n=\dim\alpha$. We see that $f(T)\leq b$ cannot be possible for the same reasons why $f(S)\leq b$ is not possible. This is a contradiction, and hence $\gamma^{(n-1)}$ is a free face of $\alpha^{(n)}$ (in $K(b)$).

It is now clear that $K(a)\subseteq K(b)\setminus\{\gamma,\alpha\}$. Let $\sigma\in K(b)\setminus\{\gamma,\alpha\}$. If $f(\sigma)\leq b$, then $f(\sigma)\leq a$ since $\sigma\neq\alpha$. Hence, $\sigma\in K(a)$. Suppose $f(\sigma)>b$, then by Lemma \ref{lemma:K(c)cond}, there exists $\tau>\sigma$ such that $f(\tau)\leq b$. If $\tau\neq\alpha$, then again $f(\tau)\leq a$ so that $\sigma\in K(a)$. If $\tau=\alpha$, by Lemma \ref{lemma:middlesimpgen} there exists $\widetilde{\alpha}^{(n-1)}$ such that $\alpha^{(n)}>\widetilde{\alpha}^{(n-1)}\geq\sigma$. If $\widetilde{\alpha}^{(n-1)}\neq\gamma^{(n-1)}$, then we have $f(\widetilde{\alpha})<f(\alpha)\leq b$. This implies that $f(\widetilde{\alpha})\leq a$ and hence $\sigma\in K(a)$. If $\widetilde{\alpha}=\gamma$, then we have $\alpha>\gamma>\sigma$ since $\sigma\neq\gamma$. By Lemma \ref{lemma:middlesimpgen}, there exists $\widetilde{\gamma}^{(n-1)}\neq\gamma$ such that $\alpha>\widetilde{\gamma}^{(n-1)}>\sigma$. Similarly, this implies that $f(\widetilde{\gamma})\leq a$ and hence $\sigma\in K(a)$. In conclusion, we have proved that $K(b)\setminus\{\gamma,\alpha\}\subseteq K(a)$ and hence $K(a)=K(b)\setminus\{\gamma,\alpha\}$.

Since $\alpha$ is $w$-simple, hence $w(\gamma)=w(\alpha)\neq 0$. By Theorem \ref{thm:sameweight}, we have that $K(b)\searrow K(a)$ is an elementary collapse that preserves weighted homology. We have completed the proof for the case $|f^{-1}((a,b])|=1$.

Finally, for the case $|f^{-1}((a,b])|=k>1$, we partition the interval $(a,b]$ into $k$ smaller subintervals
\begin{equation*}
(a,b]=(b_0,b_1]\cup(b_1,b_2]\cup\dots\cup(b_{k-1},b_k],
\end{equation*}
where $b_0=a$, $b_k=b$, such that each subinterval $(b_i,b_{i+1}]$ has exactly one simplex $\alpha_i$ with $f(\alpha_i)\in(b_i,b_{i+1}]$. Then, we have a series of elementary collapses
\begin{equation*}
K(b)\searrow K(b_{k-1})\searrow\dots\searrow K(a),
\end{equation*}
each of which preserves weighted homology. It follows that $K(b)\searrow K(a)$ is a collapse that preserves weighted homology.
\end{proof}

The next theorem generalizes Theorem 3.4 in \cite{forman1998morse} and Lemma 2.7 in \cite{forman2002user}.

\begin{theorem}[cf.\ {\cite[p.~106]{forman1998morse}}]
\label{thm:mainremoval}
Let $(K,w)$ be a WSC with a discrete Morse function $f$. Suppose $\alpha^{(n)}$ is a critical simplex with $f(\alpha)\in(a,b]$ and $f^{-1}(a,b]$ contains no other critical simplices.

Let
\begin{equation*}
a'=\min\{x\in [a,f(\alpha))\mid f^{-1}(x,f(\alpha)]=\{\alpha\}\}.
\end{equation*}
Then, we have
\begin{equation*}
K(a')=K(f(\alpha))\setminus\{\alpha\}, 
\end{equation*}
where $\alpha$ is a maximal face of $K(f(\alpha))$.

Moreover, if all simplices in
\begin{equation*}
f^{-1}(a,a']\cup f^{-1}(f(\alpha),b]
\end{equation*}
are $w$-simple, then $K(b)\searrow K(f(\alpha))$ and $K(a')\searrow K(a)$ are collapses that preserve weighted homology.
\end{theorem}
\begin{proof}
Since $K$ is finite, we note that $a'$ is well-defined and always exists. The definition of $a'$ implies that there are no simplices in $f^{-1}(a',f(\alpha)]$ other than $\alpha$.

We show that $\alpha\notin K(a')$. Suppose to the contrary there exists $\beta>\alpha$ such that $f(\beta)\leq a'$. By Lemma \ref{lemma:existsn+1}, there exists $\widetilde{\beta}^{(n+1)}$ with $\alpha<\widetilde{\beta}\leq\beta$ and
\begin{equation}
\label{eqn:criticalmid}
f(\widetilde{\beta})\leq f(\beta)\leq a'<f(\alpha).
\end{equation}
However, this contradicts the fact that $\alpha$ is critical. By Lemma \ref{lemma:K(c)cond}, we conclude that $\alpha\notin K(a')$.

Since $\alpha$ is critical, for every $\gamma^{(n-1)}<\alpha$ we have $f(\gamma)<f(\alpha)$. Since there are no simplices in $f^{-1}(a',f(\alpha)]$ other than $\alpha$, this implies $f(\gamma)\leq a'$ and hence $\gamma\in K(a')$. By Lemma \ref{lemma:middlesimpgen}, any proper face $\sigma^{(m)}<\alpha^{(n)}$ is also a face of some $\gamma^{(n-1)}<\alpha^{(n)}$. Hence, $\sigma\in K(a')$ by Lemma \ref{lemma:K(c)cond}. We have shown that all proper faces of $\alpha$ lie in $K(a')$.

We now show that $\alpha$ is a maximal face of $K(f(\alpha))$. For any $\tau>\alpha$, by arguments similar to that preceding \eqref{eqn:criticalmid}, we conclude that $f(\tau)>a'$. Since $f^{-1}(a',f(\alpha)]=\{\alpha\}$, we conclude that $f(\tau)>f(\alpha)$. This also implies that $\alpha$ cannot have a proper coface $S>\alpha$ where $S\in K(f(\alpha))$, and hence $\alpha$ is a maximal face of $K(f(\alpha))$.

It is clear that $K(a')\subseteq K(f(\alpha))\setminus\{\alpha\}$. Let $\sigma\in K(f(\alpha))\setminus\{\alpha\}$. By Lemma \ref{lemma:K(c)cond}, $f(\sigma)\leq f(\alpha)$ or there exists $\tau>\sigma$ such that $f(\tau)\leq f(\alpha)$. First, we analyze the case $f(\sigma)\leq f(\alpha)$. Since $f^{-1}(a',f(\alpha)]=\{\alpha\}$, hence $f(\sigma)\leq a'$ and $\sigma\in K(a')$. In the second case where $\tau>\sigma$ and $f(\tau)\leq f(\alpha)$, either $f(\tau)\leq a'$ (and hence $\sigma<\tau$ is in $K(a')$) or $\tau=\alpha$ (and hence $\sigma\in K(a')$ since $\sigma$ is a proper face of $\alpha$). In both cases, we have $\sigma\in K(a')$. We have shown that $K(a')=K(f(\alpha))\setminus\{\alpha\}$.

Note that all simplices in $f^{-1}(a,a']\cup f^{-1}(f(\alpha),b]$ are non-critical. Hence, if all simplices in $f^{-1}(a,a']\cup f^{-1}(f(\alpha),b]$ are $w$-simple, then by Theorem \ref{thm:maincollapse}, $K(b)\searrow K(f(\alpha))$ and $K(a')\searrow K(a)$ are collapses that preserve weighted homology.
\end{proof}

\begin{remark}
In our proofs of Theorems \ref{thm:maincollapse} and \ref{thm:mainremoval}, we do not assume (without loss of generality) that the discrete Morse function $f$ is injective, unlike the proofs in \cite{forman1998morse}.
\end{remark}

The relationship between the weighted homologies of $K(f(\alpha))$ and $K(a')$ is summarized in the following corollary.

\begin{cor}
Assume that all conditions of Theorem \ref{thm:mainremoval} hold, including the assumption that all simplices in $f^{-1}(a,a']\cup f^{-1}(f(\alpha),b]$ are $w$-simple. Moreover, suppose $w(\alpha)\neq 0$ and $R$ is a PID.

Then, we have
\begin{enumerate}
\item 
\begin{equation*}
H_k(K(a'),w)\cong H_k(K(f(\alpha)),w),
\end{equation*}
for $k\notin\{n-1,n\}$.
\item 
\begin{equation*}
H_{n-1}(K(f(\alpha)),w)\cong H_{n-1}(K(a'),w)/\{r[\partial^w\alpha]\mid r\in R\},
\end{equation*}
where $\partial^w$ denotes the weighted boundary operator of $C_*(K(f(\alpha)),w)$, and $[\partial^w\alpha]$ denotes the homology class in $H_{n-1}(K(a'),w)$.
\item
\begin{equation*}
H_n(K(f(\alpha)),w)\cong
\begin{cases}
H_n(K(a'),w)\oplus R &\text{if $[\partial^w\alpha]$ is a torsion element,}\\
H_n(K(a'),w) &\text{otherwise.}
\end{cases}
\end{equation*}
\end{enumerate}
\end{cor}
\begin{proof}
The proof is obtained by applying Theorem \ref{thm:removalhom} and Corollary \ref{cor:removalhom} to Theorem \ref{thm:mainremoval}.
\end{proof}
\section{Application to Sequence Analysis}
\label{sec:sequence}
In this section, we apply our results to the weighted homology of sequences from an alphabet $\mathcal{A}$, with a view towards DNA/RNA sequences. We emphasize that the majority of our results are equally applicable to other sequences, including binary sequences and hexadecimal sequences. We work out mathematical statements on the poset of substrings of a sequence. To our knowledge, the authors V.\ Bankston et al.\ \cite{rafal2018dna} are the first to study the topology of DNA using the sequential subword order. Our approach uses one idea of \cite{rafal2018dna}, which is to construct the poset using substrings of a sequence. The general topic of poset topology \cite{wachs2006poset} originally arose from the seminal 1964 paper of Gian-Carlo Rota on the M\"{o}bius function of a partially ordered set.

In the paper \cite{ligeti2005automorphisms}, the authors P.\ Ligeti and P.\ Sziklai study DNA using posets through another approach, focusing on the automorphism group of the posets instead of the topology.

\begin{defn}
A sequence of length $n$ over an alphabet $\mathcal{A}$ is an ordered tuple $(x_1,x_2,\dots,x_n)$, where $x_i\in\mathcal{A}$ for all $i$. For convenience, we will write the sequence $(x_1,x_2,\dots,x_n)$ as $x_1x_2\dots x_n$.
\end{defn}

\begin{eg}
A DNA sequence is a sequence over the alphabet $\mathcal{A}_\text{dna}=\{A,C,G,T\}$. A, C, G, T stand for the nitrogenous bases adenine, cytosine, guanine, and thymine respectively.
\end{eg}

\begin{eg}
An RNA sequence is a sequence over the alphabet $\mathcal{A}_\text{rna}=\{A,C,G,U\}$. U stands for the nitrogenous base uracil.
\end{eg}

\begin{eg}
A binary sequence is a sequence over the alphabet $\mathcal{A}_{\text{bin}}=\{0,1\}$. A hexadecimal sequence is a sequence over the alphabet $\mathcal{A}_\text{hex}=\{0,1,\dots,9,A,B,\dots,F\}$.
\end{eg}

\begin{defn}
Let $S=x_1x_2\dots x_n$ be a sequence over an alphabet $\mathcal{A}$. A \emph{substring} of $S$ is a sequence $T=x_{1+i}x_{2+i}\dots x_{m+i}$, where $0\leq i$ and $m+i\leq n$.
\end{defn}

\begin{remark}
We use the terminology ``substring'' as it is a well-recognized term in computer science and computational biology \cite{gusfield1997algorithms}. In \cite{rafal2018dna}, the authors use the terminology ``subword''. Note that other sources \cite{bjorner1990mobius,mcnamara2012mobius,sagan2006mobius} have different meanings for the term ``subword'' (essentially their definition of ``subword'' is that of a subsequence).
\end{remark}

\begin{eg}
Consider the DNA sequence $S=ACTGG$. Then $AC$ is a substring of $S$, but $AT$ is \emph{not} a substring of $S$. ($AT$ is a subsequence of $S$.)
\end{eg}

One of the key ideas in \cite{rafal2018dna} is to define a partial order $\preceq$ as follows:

\begin{defn}[cf.\ {\cite{rafal2018dna}}]
 We define the binary relation $\preceq$ on the set of all possible sequences over an alphabet $\mathcal{A}$ by: 
 
 $T\preceq S$ if and only if $T$ is a substring of $S$.
\end{defn}

It is clear that $\preceq$ is a partial order on the set of all possible sequences over $\mathcal{A}$. In particular, the set of substrings of a sequence $S$ is a poset. 

\begin{remark}
We remark that the similar idea of constructing a poset using subsequences instead of substrings has been studied extensively in \cite{bjorner1990mobius,mcnamara2012mobius,sagan2006mobius}. An advantage of using substrings is that a sequence $S$ typically has much fewer substrings than subsequences, hence resulting in a smaller poset which is easier to work with.
\end{remark}

Once we have a poset, the next step is to construct the order complex \cite{wachs2006poset}. Our approach differs from \cite{rafal2018dna} in that we only consider proper substrings in the poset.

\begin{defn}[{cf.\ \cite[p.~6]{wachs2006poset}}]
Let $S$ be a sequence and $(P_S,\preceq)$ be the poset of proper substrings of $S$ (i.e., excluding the substring $S$ itself and the empty substring). The \emph{order complex} of $P_S$, denoted $\Delta(P_S)$, is an abstract simplicial complex whose vertices are the substrings in $P_S$, and the faces of $\Delta(P_S)$ are the chains (i.e.\ totally ordered subsets) of $P_S$.
\end{defn}

\begin{remark}
Each distinct proper substring only appears once in $P_S$ (even if it occurs multiple times in $S$) since $P_S$ is a set. We only consider proper substrings of $S$ in $P_S$, due to the fact that a poset with a unique maximum (or minimum) element has a contractible order complex, since it is a cone (cf.\ \cite[p.~8]{wachs2006poset}). Hence, by only considering proper substrings, we get a more interesting order complex that is not necessarily contractible.
\end{remark}

\begin{eg}
\label{eg:leucine}
We consider the DNA sequence $S=CTC$, which is the DNA codon for the amino acid leucine. Then, the poset $P_S$ can be represented by the Hasse diagram in Figure \ref{fig:hasse}. The order complex $\Delta(P_S)$ is the simplicial complex shown in Figure \ref{fig:circle}, which is homotopy equivalent to a circle (not contractible).

\begin{figure}[htbp]
\begin{center}
\begin{tikzpicture}
  \node (a) at (0,0) {CT};
  \node (b) at (2,0) {TC};
  \node (d) at (0,-1) {C};
  \node (e) at (2,-1) {T};
  \draw (a)--(d)
  (b)--(e)
  (b)--(d);
  \draw[preaction={draw=white, -,line width=6pt}] (a)--(e);
\end{tikzpicture}
\caption{Hasse diagram for the poset $P_S$.}
\label{fig:hasse}
\end{center}
\end{figure}
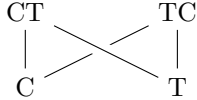

\begin{figure}[htbp]
\begin{center}
\begin{tikzpicture}
\filldraw 
(0,0) circle (2pt) node[align=left,above] {C}
(1,0) circle (2pt) node[align=left,above] {TC}  
(0,-1) circle (2pt) node[align=left,below] {CT}
(1,-1) circle (2pt) node[align=left,below] {T};
\draw (0,0)--(1,0)--(1,-1)--(0,-1)--(0,0);
\end{tikzpicture}
\caption{The order complex $\Delta(P_S)$.}
\label{fig:circle}
\end{center}
\end{figure}
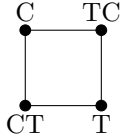
\end{eg}

\begin{remark}
\label{rem:cannotdiff}
Consider the DNA sequence $S'=GTG$, which is the DNA codon for the amino acid valine. It is clear that no topological invariant, including homology, can distinguish between the order complexes $\Delta(P_S)$ in Example \ref{eg:leucine} and $\Delta(P_{S'})$ since they only differ by a relabeling of vertices.
\end{remark}

In order to remedy the issue in Remark \ref{rem:cannotdiff}, we introduce the \emph{weighted ordered complex} (or WOC for short). For $R=\Z$, theoretically there are infinitely many weight functions $w: K\to\Z$ for a simplicial complex $K$. Hence, we emphasize that there are infinitely many different ways of defining a weighted ordered complex $(\Delta(P_S),w)$, and the most effective choice of definition may depend on the actual context of the application. We outline four such possible definitions.

\begin{defn}[Type 1 weighted ordered complex (Type 1 WOC)]
Let $S$ be a sequence over an alphabet $\mathcal{A}$. Let $(P_S,\preceq)$ be the associated poset of proper substrings of $S$. We define a weight function $w: \Delta(P_S)\to\Z$ by first assigning a weight (in $\Z$) for each letter $x\in\mathcal{A}$, and then extending to the vertices of $\Delta(P_S)$ by
\begin{equation}
\label{eq:firstlcm}
w(x_1x_2\dots x_n):=\LCM(w(x_1),\dots,w(x_n)),
\end{equation}
where $\LCM$ denotes the lowest common multiple. We then further extend $w$ to the faces of $\Delta(P_S)$ by defining
\begin{equation}
\label{eq:secondlcm}
w(\sigma):=\LCM(w(v_0),\dots,w(v_k)),
\end{equation}
for each $k$-simplex $\sigma=[v_0,v_1,\dots,v_k]$ of $\Delta(P_S)$ spanned by the vertices $v_0, \dots, v_k$.

We call $(\Delta(P_S),w)$ the \emph{weighted ordered complex} with weight function $w$.
\end{defn}

\begin{defn}[Type 2 WOC]
The definition follows that of the Type 1 weighted ordered complex, except replacing $\LCM$ with the product in Equation \eqref{eq:secondlcm}. That is, $w(\sigma)=\prod_{i=0}^k w(v_i)$.
\end{defn}

\begin{defn}[Type 3 WOC]
The definition follows that of the Type 1 WOC, except replacing $\LCM$ with the product in Equation \eqref{eq:firstlcm}. That is, $w(x_1x_2\dots x_n)=\prod_{i=1}^n w(x_i)$.
\end{defn}

\begin{defn}[Type 4 WOC]
The definition follows that of the Type 1 WOC, except replacing $\LCM$ with the product in both Equations \eqref{eq:firstlcm} and \eqref{eq:secondlcm}.
\end{defn}

It is clear that the above four types of weighted ordered complexes are weighted simplicial complexes in the sense of Definition \ref{wscdef}. 

We demonstrate that a suitable choice of weights can distinguish between certain DNA codons (i.e.\ DNA sequences of length 3) via the weighted homology of their weighted ordered complexes.

\begin{eg}
\label{eg:3amino}
Let $S=CTC$ and $S'=GTG$ be DNA sequences. Recall from Remark \ref{rem:cannotdiff} that unweighted homology is unable to tell apart their ordered complexes. Define the weights on $\mathcal{A}_\text{dna}=\{A, C, G, T\}$ by $w(A)=1$, $w(C)=2$, $w(G)=3$, and $w(T)=4$, and we extend the weight function to the faces of $\Delta(P_S)$ and $\Delta(P_{S'})$ by the definition for Type 2 WOC. For instance, we have $w(TC)=w(CT)=4$, $w([C,TC])=8$, and so on. We calculate that 
\begin{equation*}
\begin{split}
H_0(\Delta(P_S),w)&\cong \Z\oplus\Z/2\oplus\Z/2\oplus\Z/4,\\
H_0(\Delta(P_{S'}),w)&\cong \Z\oplus\Z/12.
\end{split}
\end{equation*}

Let $S''=AAA$, which is the DNA codon for the amino acid lysine. We use the same system of weights as before. The order complex of $S''$ is a 1-simplex, with weight 1 for all its faces. The 0th weighted homology of $\Delta(P_{S''})$ is therefore
\begin{equation*}
H_0(\Delta(P_{S''}),w)\cong \Z.
\end{equation*}
\end{eg}

\begin{remark}
We do not make the claim that weighted homology is able to distinguish between \emph{all} DNA sequences. In fact, with the weights defined in Example \ref{eg:3amino} (Type 2 WOC), we calculate that the 0th weighted homology of the sequence $S'''=CCT$ is the same as that of $S=CTC$. 

On the other hand, with a different choice of weights ($w'(A)=1$, $w'(C)=2$, $w'(G)=1$, and $w'(T)=3$), and extending $w'$ by the definition for Type 2 WOC, we can distinguish between the two sequences:
\begin{equation*}
\begin{split}
H_0(\Delta(P_{S}),w')&\cong \Z\oplus\Z/6,\\
H_0(\Delta(P_{S'''}),w')&\cong \Z\oplus\Z/2\oplus\Z/6.
\end{split}
\end{equation*}
\end{remark}

\subsection{Discrete Morse Theory and Weighted Ordered Complexes}
We apply our previous results on weighted discrete Morse theory to the study of weighted ordered complexes. We study some examples where weighted discrete Morse theory can be applied to calculate weighted homology.

\begin{prop}
Let $x$, $y$ be distinct letters in an alphabet $\mathcal{A}$. Consider the sequence $S=xyyy$. Let $w(x)=a$, $w(y)=b$, where $a,b\in\Z_{>0}$. We extend $w$ to $\Delta(P_S)$ by the definition of Type 3 WOC. Then,

\begin{equation*}
H_i(\Delta(P_S),w)\cong\begin{cases}
\Z\oplus\Z/\gcd(a,b) &\text{if $i=0$},\\
0 &\text{if $i\geq 1$}.
\end{cases}
\end{equation*}
\end{prop}
\begin{proof}
The order complex $K=\Delta(P_S)$ is shown in Figure \ref{fig:ocxyyy}.

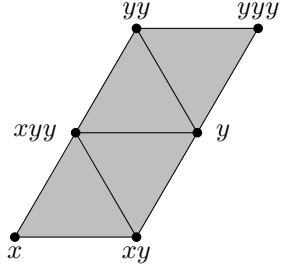
\begin{figure}[htbp]
\begin{tikzpicture}[scale=0.8]
\fill[fill=lightgray]
(0,0)  
-- (2,0)
-- (1,1.7);
\fill[fill=lightgray]
(1,1.73)  
-- (3,1.73)
-- (2,0);
\fill[fill=lightgray]
(1,1.73)  
-- (3,1.73)
-- (2,3.46);
\fill[fill=lightgray]
(4,3.46)  
-- (3,1.73)
-- (2,3.46);
\filldraw 
(0,0) circle (2pt) node[align=left,below] {$x$}
(2,0) circle (2pt) node[align=left,below] {$xy$}  
(1,1.73) circle (2pt) node[label=left:$xyy$] {}
(3,1.73) circle (2pt) node[label=right:$y$] {}
(2,3.46) circle (2pt) node[align=left,above] {$yy$}
(4,3.46) circle (2pt) node[align=left,above] {$yyy$};
\draw (0,0)--(2,0);
\draw (2,0)--(1,1.73);
\draw (1,1.73)--(0,0);
\draw (1,1.73)--(3,1.73)
(3,1.73)--(2,0);
\draw (1,1.73)--(2,3.46)
(2,3.46)--(3,1.73);
\draw (2,3.46)--(4,3.46)
(3,1.73)--(4,3.46);
\end{tikzpicture}
\caption{The order complex $K=\Delta(P_S)$.}
\label{fig:ocxyyy}
\end{figure}

Consider the discrete Morse function $f: K\to\R$ as follows:
\begin{itemize}
\item Simplices $\sigma$ such that $f(\sigma)=1$: $[x]$, $[xy]$, $[y]$
\item $\sigma$ such that $f(\sigma)=2$: $[x,xy]$, $[xy,y]$
\item $\sigma$ such that $f(\sigma)=3$: $[xyy]$, $[yy]$, $[xy,xyy]$, $[y,yy]$
\item $\sigma$ such that $f(\sigma)=4$: $[yyy]$, $[yy,yyy]$, $[xyy,y]$, $[y,xy,xyy]$
\item $\sigma$ such that $f(\sigma)=5$: $[x,xyy]$, $[xyy,yy]$, $[y,yyy]$, $[x,xy,xyy]$, $[y,yy,xyy]$, $[y,yy,yyy]$.
\end{itemize}

We note that all simplices in $f^{-1}((2,5])$ are non-critical and $w$-simple. Hence, by Theorem \ref{thm:maincollapse}, $K=K(5)\searrow K(2)$ is a collapse that preserves weighted homology. The level subcomplex $K(2)$ is shown in Figure \ref{fig:K2}.

\begin{figure}[htbp]
\begin{tikzpicture}[scale=0.8]
\filldraw 
(0,0) circle (2pt) node[align=left,below] {$x$}
(2,0) circle (2pt) node[align=left,below] {$xy$}  
(3,1.73) circle (2pt) node[label=right:$y$] {};
\draw (0,0)--(2,0)--(3,1.73);
\end{tikzpicture}
\caption{The level subcomplex $K(2)$.}
\label{fig:K2}
\end{figure}
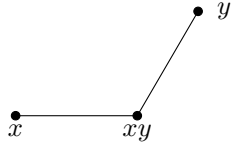

Hence, by observing $K(2)$, it is clear that $H_i(K,w)\cong H_i(K(2),w)\cong 0$ for $i\geq 1$. We calculate that
\begin{equation*}
\begin{split}
\partial_1^w ([x,xy])&=[xy]-b[x]\\
\partial_1^w ([xy,y])&=a[y]-[xy].
\end{split}
\end{equation*}

The matrix for the weighted boundary operator $\partial_1^w$ can be written as 
\begin{equation*}
[\partial_1^w]=\begin{pmatrix}-b &0\\1 &-1\\ 0& a\end{pmatrix}. 
\end{equation*}
By Prop.\ 8.1 in \cite{miller2009differential}, the nonzero diagonal entries for the Smith normal form of $[\partial_1^w]$ are 1 and $\gcd(a,b)$. Hence, we may conclude that $H_0(K,w)\cong\Z\oplus\Z/\gcd(a,b)$.
\end{proof}

\begin{prop}
Let $x$ be a letter in an alphabet $\mathcal{A}$. Consider the sequence $S=\underbrace{xx\dots x}_\text{$n$ times}$. Let $w(x)=a\in\Z\setminus\{0\}$. We extend $w$ to $\Delta(P_S)$ by the definition of Type 3 WOC. Then,
\begin{equation*}
H_i(\Delta(P_S),w)\cong\begin{cases}
\Z &\text{if $i=0$},\\
0 &\text{if $i\geq 1$}.
\end{cases}
\end{equation*}
\end{prop}
\begin{proof}
We note that the order complex $K=\Delta(P_S)$ is a $(n-2)$-simplex with the $n-1$ vertices $v_1=x$, $v_2=xx$, $\dots$, $v_{n-1}=\underbrace{xx\dots x}_\text{$n-1$ times}$.

Consider the discrete Morse function $f: K\to\R$ defined as follows. We let $f([v_1])=1$. Let $A_k$ be the set of $k$-simplices that do not contain the vertex $v_1$, where $0\leq k\leq n-3$. We arrange $A_k$ in lexicographical order and define $f([v_2])=2$, $f([v_3])=3$, $\dots$, $f([v_{n-1}])=n-1$. We now proceed to define $f$ on $A_k$, in ascending order of $k$. Similarly, we arrange $A_k$ in lexicographical order and define the value of $f$ on each $k$-simplex to be the smallest integer value that has not been used so far in the definition of $f$. Now, let
\begin{equation*}
B=\{[v_1,v_{i_0},v_{i_1},\dots,v_{i_l}]\mid [v_{i_0},v_{i_1},\dots,v_{i_l}]\in A_l\ \text{for some $0\leq l\leq n-3$}\}.
\end{equation*}
For $\sigma=[v_1,v_{i_0},\dots,v_{i_l}]\in B$, we define 
\begin{equation}
\label{eq:abpair}
f(\sigma)=f([v_{i_0},\dots,v_{i_l}]).
\end{equation}

It can be verified that $f$ is indeed a discrete Morse function. If $\alpha=[v_1]$, then for any $\beta=[v_1,v_j]$, we have $f(\beta)=f([v_j])>f([v_1])$. If $\alpha^{(k)}\in A_k$, then any $\beta^{(k+1)}>\alpha^{(k)}$ is either in $A_{k+1}$ or in $B$. If $\beta\in A_{k+1}$, then $f(\beta)>f(\alpha)$ by the definition of $f$. If $\beta\in B$, then $\beta$ is uniquely determined and $f(\beta)=f(\alpha)$. We note that any $\gamma^{(k-1)}<\alpha^{(k)}$ must be in $A_{k-1}$, and hence $f(\gamma)<f(\alpha)$ by the definition of $f$. A similar analysis follows in the case of $\alpha\in B$.

Let $b=f([v_1,\dots,v_{n-1}])$. Note that all simplices in $f^{-1}((1,b])$, i.e.\ all simplices other than $[v_1]$, are either in $A_k$ for some $k$ or in $B$. By Equation \ref{eq:abpair}, we see that each simplex in $f^{-1}((1,b])$ is non-critical.

Let $\alpha^{(k)}\in f^{-1}((1,b])$. Suppose $\gamma^{(k-1)}<\alpha$ and $f(\gamma)\geq f(\alpha)$. By the definition of $f$, it implies that $\alpha=[v_1,v_{i_0},\dots,v_{i_l}]\in B$ and $\gamma=[v_{i_0},\dots,v_{i_l}]\in A_l$. By the definition of Type 3 WOC, we have
\begin{equation*}
w(\gamma)=w(\alpha)=\LCM(w(v_{i_0}),\dots,w(v_{i_l})).
\end{equation*}
Therefore, all simplices in $f^{-1}((1,b])$ are $w$-simple.

Hence, by Theorem \ref{thm:maincollapse}, $K=K(b)\searrow K(1)$ is a collapse that preserves weighted homology. We observe that $K(1)=\{[v_1]\}$, and hence the result follows.
\end{proof}

\section*{Acknowledgements}
We would like to thank Professor Rafal Komendarczyk for introducing his work \cite{rafal2018dna} to us through private communication. We wish to thank the referees most warmly for numerous suggestions that have improved the exposition of this paper.

%%%%%Main Text %%%%%
\bibliographystyle{amsplain}
%\nocite{*}
\bibliography{jabref9}

\end{document}